\documentclass[a4paper,12pt]{amsart}

\usepackage{amssymb, amsmath, amsthm, mathrsfs, braket, xspace}
\usepackage[margin=1.0in]{geometry}

\numberwithin{equation}{section}
\usepackage[colorlinks=true]{hyperref}

\usepackage{enumitem}
\setlist[enumerate]{itemsep=0pt,label=$(\mathrm{\roman*})$, topsep=5pt}
\setlist[itemize]{itemsep=0pt, topsep=5pt, labelindent=\parindent, leftmargin=*}
\setlist[description]{itemsep=0pt, topsep=5pt, leftmargin=*}

\usepackage{mathtools}
\mathtoolsset{showonlyrefs=true} 

\usepackage[all,2cell]{xy}
\objectmargin+{1mm}
\labelmargin+{0.8mm}
\SelectTips{cm}{12}

\usepackage{aliascnt}

\usepackage[OT2,T1]{fontenc}
\DeclareSymbolFont{cyrletters}{OT2}{wncyr}{m}{n}
\DeclareMathSymbol{\Sha}{\mathalpha}{cyrletters}{"58}

\newtheorem{thm}{Theorem}[section]

\newaliascnt{cor}{thm}
\newtheorem{cor}[cor]{Corollary}
\aliascntresetthe{cor}

\newaliascnt{lem}{thm}
\newtheorem{lem}[lem]{Lemma}
\aliascntresetthe{lem}

\newaliascnt{prop}{thm}
\newtheorem{prop}[prop]{Proposition}
\aliascntresetthe{prop}

\newaliascnt{conj}{thm}

\aliascntresetthe{conj}

\theoremstyle{definition}
\newaliascnt{dfn}{thm}

\aliascntresetthe{dfn}

\newaliascnt{rem}{thm}
\newtheorem{rem}[rem]{Remark}
\aliascntresetthe{rem}

\newaliascnt{prob}{thm}

\aliascntresetthe{prob}

\newaliascnt{ex}{thm}
\newtheorem{ex}[ex]{Example}
\aliascntresetthe{ex}

\newcommand{\ab}{\mathrm{ab}}

\DeclareMathOperator{\Aut}{Aut}
\newcommand{\bad}{\mathrm{bad}}
\newcommand{\C}{\mathbb{C}}
\newcommand{\Char}{\operatorname{char}}

\DeclareMathOperator{\Coker}{Coker}
\renewcommand{\div}{\mathrm{div}}
\newcommand{\dimFl}{\dim_{\Fl}}
\renewcommand{\d}{\partial}
\newcommand{\dbar}{\ol{\d}}

\newcommand{\dt}{\d^{\,\mathrm{tame}}}
\newcommand{\ds}{\displaystyle}
\newcommand{\E}{\mathscr{E}}

\newcommand{\Ebar}{\ol{E}}

\newcommand{\End}{\operatorname{End}}
\newcommand{\F}{\mathbb{F}}
\newcommand{\Fl}{\F_l}
\newcommand{\Fv}{\F_v}
\newcommand{\Frob}{\operatorname{Frob}}
\newcommand{\fin}{\mathrm{fin}}
\newcommand{\geo}{\mathrm{geo}}
\DeclareMathOperator{\Gal}{Gal}
\newcommand{\Gm}{\mathbb{G}_{m}}
\newcommand{\good}{\mathrm{good}}
\newcommand{\Hom}{\operatorname{Hom}}
\newcommand{\isomto}{\xrightarrow{\simeq}}
\renewcommand{\Im}{\operatorname{Im}}
\newcommand{\inj}{\hookrightarrow}
\DeclareMathOperator{\Jac}{Jac}
\newcommand{\Jbar}{\ol{J}}
\newcommand{\J}{\mathscr{J}}
\newcommand{\Jvbar}{\ol{J}_v}
\DeclareMathOperator{\Ker}{Ker}
\newcommand{\loc}{\operatorname{loc}}
\newcommand{\locbar}{\ol{\loc}}

\newcommand{\m}{\mathfrak{m}}
\newcommand{\M}{\mathscr{M}}
\newcommand{\N}{\mathscr{N}}
\newcommand{\ol}[1]{\overline{#1}}
\renewcommand{\O}{\mathcal{O}}

\newcommand{\otimesM}{\overset{M}{\otimes}}

\newcommand{\piab}{\pi_1^{\ab}}
\newcommand{\piabX}{\piab(X)}

\newcommand{\piabXgeo}{\piabX^{\geo}}

\newcommand{\piabXgeoram}{\piabXgeo_{\ram}}
\newcommand{\plim}{\varprojlim}
\newcommand{\p}{\mathfrak{p}}
\newcommand{\pmat}[1]{\begin{pmatrix}
	#1
\end{pmatrix}}
\newcommand{\Q}{\mathbb{Q}}
\newcommand{\Qp}{\Q_p}
\newcommand{\QZ}{\Q/\Z}

\newcommand{\R}{\mathbb{R}}
\newcommand{\Res}{\operatorname{Res}}
\newcommand{\res}{\operatorname{res}}
\newcommand{\ram}{\mathrm{ram}}
\newcommand{\red}{\mathrm{red}}
\newcommand{\rhobar}{\ol{\rho}}

\newcommand{\sep}{\mathrm{sep}} 

\newcommand{\ssm}{\smallsetminus}
\newcommand{\Spec}{\operatorname{Spec}}
\newcommand{\tor}{\mathrm{tor}}
\newcommand{\Sp}{\operatorname{Sp}}
\newcommand{\ubar}{\ol{u}}
\newcommand{\VX}{V(X)}

\newcommand{\X}{\mathscr{X}}
\newcommand{\Xbar}{\ol{X}}
\newcommand{\Xvbar}{\Xbar_v}
\newcommand{\Z}{\mathbb{Z}}

\title[A Hasse principle for higher Chow groups of curves]{A Hasse principle for higher Chow groups of curves over global fields}

\author[T. Hiranouchi]{Toshiro Hiranouchi}\address[T. Hiranouchi]{
Department of Basic Sciences, Graduate School of Engineering, 
Kyushu Institute of Technology, 
1-1 Sensui-cho, Tobata-ku, Kitakyushu-shi, 
Fukuoka 804-8550 JAPAN}
\email{hira@mns.kyutech.ac.jp}

\date{}

\begin{document}
\pagenumbering{arabic}
\begin{abstract}
Let $X$ be a smooth projective curve over a global field $F$, and let
$V(X)$ denote the kernel of the push-forward map
$CH^2(X,1)\to CH^1(F,1) = F^\times$ of the higher Chow groups.
We describe the structure of $V(X)/lV(X)$ by combining Bloch's exact
sequence with a Hasse principle in Galois cohomology for the mod-$l$
Galois representation of the Jacobian variety $J$ of $X$.
We also give explicit computations for elliptic curves.
\end{abstract}
\maketitle

\section{Introduction}
Let $F$ be a \textbf{global field}, that is, a finite extension of $\Q$ or a function field of one variable over a finite field. 
We denote by $\Char(F)$ the characteristic of $F$. 
Let $X$ be a smooth projective curve over
$F$ with $X(F)\neq\emptyset$. In this paper, we study the kernel  
\begin{equation}
	\label{def:VX-intro}
V(X):=\Ker\bigl(f_*\colon CH^2(X,1)\to CH^1(F,1)=F^\times\bigr),
\end{equation}
of the push-forward map on the higher Chow groups, 
where $f\colon X\to\Spec(F)$ is the structure morphism. A rational point
of $X$ splits $f_*$, and hence
\[
CH^2(X,1)\simeq F^\times\oplus V(X).
\]
Bloch conjectured that $V(X)$ is torsion (\cite[Remark~1.24]{Blo81}). 
\color{black}At present, the following partial results are known: 
\begin{itemize}
	\item If $\Char(F) = 0$, then $V(X)$ is the direct sum of a uniquely divisible group and a torsion group  (\cite[Theorem~0.2]{Ras90}). 
	\item If $\Char(F)>0$ and $X=X_0\otimes_{\F}F$ for a smooth
	projective curve $X_0$ over the finite constant field $\F\subset F$,
	then $V(X)$ is torsion (\cite[Corollary 3.4]{Akh05}).
\end{itemize}
\color{black}
In this paper, we study $V(X)/lV(X)$ for primes $l\neq \Char(F)$, 
in terms of the reduction of $X$ at the places of $F$
and the Galois action on $J[l]$, where $J=\Jac_X$ is the Jacobian variety of $X$.

After fixing an $F$-rational point $x_0\in X(F)$, there is a canonical isomorphism
\[
V(X)\simeq K(F;J,\Gm),
\]
where 
the right-hand side is the Somekawa $K$-group
attached to $J$ and $\Gm$ (\cite[Theorem~2.1]{Som90} and \cite[Section~2]{RS00}). 
\color{black}
If we replace $J$ with $\Gm$, the Somekawa $K$-group  
$K(F;\Gm,\Gm)$ is isomorphic to the Milnor $K$-group $K_2^M(F)$ (\cite[Theorem~1.4]{Som90}): 
\[
K_2^M(F) \simeq K(F;\Gm,\Gm).
\]
The structure of $K_2^M(F)$ can be studied via  
the tame symbol map 
\[
	\d_F\colon K_2^M(F) \to \bigoplus_{v\colon\text{finite place of $F$}} \F_v^{\times} .
\]
Here, 
$\F_v$ is the residue field of $F$ at a finite place $v$ of $F$. 
Over a number field, the tame symbol map $\d_F$ is surjective and has finite kernel (\cite[Theorem~7]{Wei05}).
To study the structure of the group $V(X)$, 
we consider the map of the form 
\[
\dbar_{X,l}^{\good}\colon V(X)/lV(X)\to
\bigoplus_{v\colon\text{finite good place of $F$}}\Jbar_v(\F_v)/l\Jbar_v(\F_v),
\]
where 
$\Jbar_v$ denotes the reduction at $v$ of the Jacobian variety $J$ 
(see \autoref{subsec:main} for the precise definition).  
\color{black}
Our aim is to describe the kernel and cokernel of this map in terms of
local reduction data and the $G_F$-coinvariant quotient $J[l]_{G_F}$.

To state our result more precisely, 
for an odd prime $l\neq \Char(F)$, 
let
\[
\rhobar_{J,l}\colon G_F= \Gal(F^{\sep}/F) \to\Aut(J[l])
\]
be the mod-$l$ Galois representation induced by the natural action of $G_F$ on $J[l]$, and put
$G_l:=\Im(\rhobar_{J,l})$ the image of $\rhobar_{J,l}$.
\color{black}
Put
\[
\Sha^1(G_l,J[l]^\vee):=
\Ker\left(
H^1(G_l,J[l]^\vee)\to
\prod_{v\colon\text{place in }F}
H^1((G_l)_v,J[l]^\vee)
\right),
\]
where $J[l]^\vee=\Hom(J[l],\Q/\Z)$ \color{black} and $(G_l)_v$ is the image of a
decomposition group at $v$.
For each finite place $v$ of $F$, 
let $F_v$ be the completion of $F$ at $v$, and put 
$X_v = X\otimes_F F_v$, 
and let $\Xvbar$ be the reduction of $X_v$. 
If $X$ has good reduction at $v$, then the local boundary map
\[
\d_{X_v}\colon CH^2(X_v,1) \to CH_0(\Xvbar)
\]
of the higher Chow groups induces 
\[
\dbar_{X_v,l}\colon V(X_v)/lV(X_v) \to \Jvbar(\Fv)/l\Jvbar(\Fv), 
\]
where $V(X_v) \subset CH^2(X_v,1)$ is defined as in \eqref{def:VX-intro}.
In fact, the map $\dbar_{X_v,l}$ is surjective, and its kernel is described by 
the class field theory of the curve $X_v$ (see \autoref{sec:local}).
Define
\[
V(X,l):=
\begin{cases}
\displaystyle
\bigoplus_{\substack{v \colon \text{finite good place in $F$}\\ v\mid l}}
\Ker(\dbar_{X_v,l})
\oplus
\bigoplus_{v\colon \text{bad place in $F$}}V(X_v)/lV(X_v),
& \Char(F)=0,\\
\displaystyle
\bigoplus_{v\colon \text{bad place in $F$}}V(X_v)/lV(X_v),
& \Char(F)>0.
\end{cases}
\]
The following is the main structural result of $V(X)/lV(X)$.

\begin{thm}[{\autoref{thm:pdiv} and \autoref{thm:main}}]
\label{thm:main_intro}
Let $l$ be an odd prime with $l\neq\Char(F)$. If
$\Sha^1(G_l,J[l]^\vee)=0$,
then the localization map
\[
V(X)/lV(X)\to\bigoplus_{v\colon\mathrm{place\,in}\,F}V(X_v)/lV(X_v)
\]
is injective, and there is an exact sequence
\[
0\to\Ker(\dbar_{X,l}^{\good})
\to V(X,l)
\to J[l]_{G_F}
\to\Coker(\dbar_{X,l}^{\good})
\to0,
\]
where $J[l]_{G_F}$ is the maximal quotient of $J[l]$ on which $G_F$ acts trivially. 
\end{thm}
The cohomological hypothesis has a simple sufficient condition: it
holds when the image of $\rhobar_{J,l}$ is trivial or contains a
nontrivial scalar matrix. Combining this criterion with big-monodromy
theorems gives the following consequence for curves of arbitrary genus.
\begin{thm}[{\autoref{cor:almostall}}]
\label{thm:main_intro2}
Assume that
$\End_{\overline F}(J)=\Z$
and that $J$ has semistable reduction of toric dimension one at some finite
place of $F$. Then
\[
\dbar_{X,l}^{\good}\colon V(X)/lV(X)\to
\bigoplus_{v\colon \rm{finite\,good\,place\,of}\,F}\Jvbar(\Fv)/l\Jvbar(\Fv)
\]
is an isomorphism for all but finitely many primes
$l\neq\Char(F)$.
\end{thm}
For elliptic curves, the cohomological condition in \autoref{thm:main_intro} is automatic by \autoref{prop:Sha}, so the exact sequence of \autoref{thm:main_intro} holds without any assumption on the image of the mod-$l$ representation.
Our earlier paper \cite[Theorem~1.1]{Hir25} studied $V(E)/lV(E)$ for elliptic curves
over number fields, with applications to several exceptional primes
$l$. In the present paper, the elliptic-curve case is treated over
arbitrary global fields for any odd prime $l\neq \Char(F)$. 

The paper is organized as follows. Section~2 recalls Mackey products,
Somekawa $K$-groups, and the description of $V(X)$ in terms of the
abelian fundamental group. Section~3 studies local boundary maps for
curves. Section~4 establishes the global Hasse principle and the main
exact sequence, and derives the almost-all-$l$ theorem. Section~5 gives
the local computations for elliptic curves with multiplicative
reduction. Section~6 contains explicit examples.

\subsection*{Notation}
We use standard notation for local and global fields. For a field $F$, write $\Char(F)$ for its characteristic and $G_F=\Gal(F^{\sep}/F)$. For a Galois extension $L/F$ and a $\Gal(L/F)$-module $M$, let $H^i(L/F,M)$ denote continuous Galois cohomology, and write $H^i(F,M)$ when $L=F^{\sep}$. For a global field $F$, let $P(F)$ be its set of places, $F_v$ the completion at $v$, and $\F_v$ the residue field when $v$ is finite. For a local field $K$, write $v_K$, $\O_K$, $\m_K$, and $\F_K$ for its normalized valuation, valuation ring, maximal ideal, and residue field. For an abelian group $A$, write $A[m]$, $A/m$, $A_\div$, and $A_\red=A/A_\div$ in the usual sense; for an extension $L/F$, put $X_L=X\otimes_F L$.

\subsection*{Acknowledgements}
The author was supported by JSPS KAKENHI Grant Number 24K06672. 
The author thanks the referee for the careful reading and helpful comments.
\section{Preliminary}

\subsection{Mackey products}
We use Mackey products and Somekawa $K$-groups in the sense of
\cite[Section~1]{Kah92a} and \cite[Section~1]{Som90}. If $\M$ and $\N$ are Mackey functors over $F$, the Mackey product $\M\otimesM\N$ is generated by symbols $\{x,y\}_{L/F}$, for finite extensions $L/F$, subject to the projection formulas
\[
\{N_{L/K}x,y\}_{K/F}=\{x,\Res_{L/K}y\}_{L/F},\qquad
\{x,N_{L/K}y\}_{K/F}=\{\Res_{L/K}x,y\}_{L/F}.
\]
For a semiabelian variety $G$ and an integer $m$ prime to $\Char(F)$, the Kummer sequence gives an injection
\begin{equation}\label{def:Kummer}
\delta_G\colon G(L)/m\inj H^1(L,G[m]).
\end{equation}
For semiabelian varieties $G_1,G_2$, cup product and corestriction define the Galois symbol
\begin{equation}\label{def:s}
s_m\colon K(F;G_1,G_2)/m\to H^2(F,G_1[m]\otimes G_2[m]),
\end{equation}
characterized by
\[
s_m(\{x_1,x_2\}_{L/F})=
\mathrm{Cor}_{L/F}\bigl(\delta_{G_1}(x_1)\cup\delta_{G_2}(x_2)\bigr).
\]
Here $K(F;G_1,G_2)$ is the quotient of $(G_1\otimesM G_2)(F)$ by the Weil reciprocity relations of \cite[Section~1]{Som90}.

\subsection{Abelian fundamental groups for curves}
\color{black}Let $X$ be a smooth projective geometrically connected curve over
a field $F$, and assume that $X(F)\ne\emptyset$. We denote by $X_0$ the set of closed points of $X$. \color{black} Fix $x_0\in X(F)$,
and put $J=\Jac_X$.  
The group $SK_1(X)$ is defined by the cokernel of the tame symbol map 
\[
SK_1(X) = \Coker\left(\dt_{F(X)} \colon K_2^M(F(X)) \to \bigoplus_{x\in X_0} F(x)^{\times}\right), 
\]
where  
$\dt_{F(X)}$ is the direct sum of the tame symbols at the closed points of $X$, 
$F(x)$ is the residue field at $x\in X_0$, and $F(X)$ is the function
field of $X$. 
The norm maps $N_{F(x)/F}\colon F(x)^{\times} \to F^{\times}$ 
for closed points $x \in X_0$ 
induce $N\colon SK_1(X) \to F^{\times}$; 
set $\VX := \Ker(N)$. 
The point $x_0$ gives a section of $N$; hence the short exact sequence
\[
0\to V(X) \to SK_1(X)\xrightarrow{N} F^{\times} \to 0
\]
splits. 
The Milnor-type $K$-group $K(F;J,\Gm)$ 
associated with the Jacobian variety $J := \Jac_X$ of $X$ and the multiplicative group $\Gm$ 
is generated by symbols $\set{P,f}_{F'/F}$ 
of $P\in J(F')$ and $f\in \Gm(F') = (F')^{\times}$ 
for a finite field extension $F'/F$. 
It follows from \cite[Theorem~2.1]{Som90} that 
there is a canonical isomorphism 
\begin{equation}
\label{eq:Som}	
 V(X)\isomto K(F;J,\Gm).
\end{equation}
The point $x_0$ also gives a splitting of the exact sequence
\[
0 \to \piabXgeo\to \piabX\to G_F^{\ab}\to 0
\]
of abelian fundamental groups, where
$G_F^{\ab} = \Gal(F^{\ab}/F)$ is the Galois group of the maximal abelian extension $F^{\ab}$ of $F$, and 
$\piabXgeo$ denotes the kernel of $\piabX\to G_F^{\ab}$. \color{black}For a discrete torsion abelian group $A$, its Pontryagin dual is $A^\vee=\Hom_{\rm cont}(A,\Q/\Z)$. \color{black}
The geometric part $\piabXgeo$ is naturally identified with the
$G_F$-coinvariant quotient $T(X)_{G_F}$, where 
\begin{equation}\label{eq:TX}
T(X) := H^1(X_{F^{\sep}},\QZ)^{\vee}.
\end{equation}
There is a decomposition 
$T(X) = \prod_{l\colon \text{prime}} T_l(X)$, where  
$T_l(X) := H^1(X_{F^{\sep}},\Q_l/\Z_l)^{\vee}$.
For a prime $l \neq \Char(F)$, 
$T_l(X) \simeq T_l(J)$
where $T_l(J) = \plim_n J[l^n](F^{\sep})$ 
is the $l$-adic Tate module of the Jacobian variety $J$
(cf.~\cite[Section~1]{KL81} and \cite[Section~3]{KS83b}).
For every positive integer $m$ prime to $\Char(F)$, 
the \textbf{Galois symbol map} 
\begin{equation}
	\label{eq:sFl}
	s_{X,m}\colon V(X)/m \simeq K(F;J,\Gm)/m \hookrightarrow H^2(F,J[m](1)) = H^2(F,J[m]\otimes \mu_m)
\end{equation}
is injective, where $\mu_m$ is the group of $m$-th roots of unity (\cite[Theorem~6.1]{Yam05}).

\section{Local boundary maps for curves}\label{sec:local}
Throughout this section, let $K$ be a local field with finite residue field $\F = \F_K$. 
Put $p := \Char(\F)$. 
\color{black}
We write $v_K\colon K^\times\to\Z$ for the normalized valuation of $K$,
and $\O_K$ and $\m_K$ for its valuation ring and maximal ideal, respectively. 
\color{black}
Let $X$ be a smooth, projective, geometrically irreducible curve over
$K$ with $X(K)\neq\emptyset$. 
We use the following notation:
\begin{itemize}
\item $\X$: a proper regular model of $X$, that is,  a proper, regular, flat
scheme over $\O_K$ whose generic fiber is $X$ (cf.~\cite[Chapter~IV, Theorem~4.5]{Sil151}). 
\item $\Xbar:= \X\otimes_{\O_K}\F$:  the special fiber of $\X$.
\end{itemize}

\subsection{Class field theory}
Following \cite{Blo81}, \cite{Sai85a},
\cite{KS83b}, and \cite{Yos03}, let
\[
\tau_X\colon V(X)\to\piab(X)^{\geo}
\]
be the geometric part of the reciprocity map for $X$.
\begin{thm}[{\cite[Chapter II, Theorem~4.1, Theorem~5.1]{Sai85a},
\cite[Theorem~5.1]{Yos03}}]\label{thm:BS}
The kernel of $\tau_X$ is $V(X)_\div$, and its image is finite. Consequently,
\[
V(X)_\red\simeq\piab(X)^{\geo}_{\tor}
\]
is finite.
\end{thm}
Let $\X$ be a proper regular model of $X$ over $\O_K$, with special fiber $\Xbar$. \color{black}For the proper curve $\Xbar$ over $\F$, let $A_0(\Xbar)=\Ker\left(\deg\colon CH_0(\Xbar)\to\Z\right)$. The special fiber $\Xbar$ is connected (\cite[\href{https://stacks.math.columbia.edu/tag/0AY8}{Lemma 0AY8}]{stacks-project}). \color{black} The valuation maps at the closed points induce a surjective boundary map
\[
\d_X\colon SK_1(X)\twoheadrightarrow CH_0(\Xbar),
\]
and hence a surjective map $\d_X\colon V(X)\twoheadrightarrow A_0(\Xbar)$. 
\color{black}Indeed, this follows from the splitting given by $x_0$ and the compatibility 
$\deg\circ \d_X = v_K\circ N$.
\color{black}
It is compatible with the specialization map on abelian fundamental groups. Since $A_0(\Xbar)\simeq\piab(\Xbar)^{\geo}_{\tor}$ by class field theory over finite fields \color{black}(\cite[Section~1, Corollary to Proposition~1]{KS83b})\color{black}, this compatibility gives an injection
\begin{equation}\label{eq:ker-ram}
\Ker\bigl((\d_X)_\red\colon V(X)_\red\to A_0(\Xbar)\bigr)
\hookrightarrow
\piabXgeoram,
\end{equation}
where
\[
\piabXgeoram:=\Ker\bigl(\piabXgeo\to\piab(\Xbar)^{\geo}\bigr).
\]
For the residue characteristic $p = \Char(\F)$, we denote by 
\[
(\piabXgeoram)^{(p')} \quad \mbox{and}\quad (\piabXgeoram)^{(p)}
\]
the prime-to-$p$ part and the $p$-primary part, respectively, of the
finite group $\piabXgeoram$.
Let $J = \Jac_X$ be the Jacobian variety of $X$. 
Let $\J^N$ be the N\'eron model of $J$, let $\Jbar^N$ be its special
fiber, and let $(\Jbar^N)^0$ be the identity component. 
The quotient   
\[
\Phi_J := \Jbar^N/(\Jbar^N)^0
\]
is a finite \'etale group over $\F$. 
\begin{thm}[{\cite[Theorem~3.2, Theorem~4.1]{Yos02}}]\label{thm:Yos}
	Put $p = \Char(\F)$.
	\begin{enumerate}
		\item There is an injection of finite abelian groups 
		\[
		{(\piabXgeoram)^{(p')}}^\vee \hookrightarrow\Phi_J(\ol{\F}),
		\]
		where $(-)^\vee$ denotes the Pontryagin dual. 
		\item If $\Char(K) = 0$, $X$ has semistable reduction, and the absolute ramification index satisfies $e_{K/\Qp} <p-1$, then 
		we have 
		\[
		(\piabXgeoram)^{(p)} = 0.
		\]
	\end{enumerate}
\end{thm}
For every prime $l$, 
the boundary map $\d_X \colon V(X) \twoheadrightarrow A_0(\Xbar)$ induces 
a homomorphism
\begin{equation}\label{eq:dbarlocal}
\dbar_{X,l}\colon V(X)/l \twoheadrightarrow A_0(\Xbar)/l.
\end{equation}
Recall that $X$ has \textbf{good reduction} if the special fiber
$\Xbar$ is smooth over $\F$. 
The following lemma describes the kernel of $\dbar_{X,l}$.
\begin{lem}\label{lem:local0}
	Let $l$ be a prime. 
	Assume that one of the following conditions holds:  
	\begin{enumerate}[label=$(\alph*)$]
		\item $X$ has semistable reduction, $\Char(K) =0$, $l=p$, and $e_{K/\Qp}<p-1$. 
		\item $X$ has good reduction, and $l\neq p$.
	\end{enumerate}
	Then $\Ker(\dbar_{X,l}) =0$. 
\end{lem}
\begin{proof}
By \eqref{eq:ker-ram}, there is an injection
\[
\Ker((\d_X)_{\red})\hookrightarrow\piabXgeoram.
\]
Since $V(X)_\div$ is divisible and $A_0(\Xbar)$ is finite, $\d_X$
factors through $V(X)_\red$. Consequently, there is a natural
surjection
\[
\Ker((\d_X)_{\red})/l\twoheadrightarrow\Ker(\dbar_{X,l}).
\]
Because $\Ker((\d_X)_{\red})$ is finite, its quotient modulo $l$ is a
quotient of its $l$-primary part. Under hypothesis (a),
\autoref{thm:Yos}(ii) gives $(\piabXgeoram)^{(l)}=0$. Under hypothesis
(b), the Jacobian $J$ has good reduction, so
$\Phi_J(\ol{\F})=0$; hence \autoref{thm:Yos}(i) again gives
$(\piabXgeoram)^{(l)}=0$. Thus the $l$-primary part of
$\Ker((\d_X)_{\red})$ is trivial, and therefore
$\Ker(\dbar_{X,l})=0$.
\end{proof}

\section{Global fields}
\label{sec:global}
Throughout this section, let $X$ be a smooth projective curve over a
global field $F$ with $X(F)\neq\emptyset$.
Let $J = \Jac_X$ be the Jacobian variety of $X$ and put $g = \dim(J)$. 
If $\Char(F)=0$, we write
$P_\fin(F)$ and $P_\infty(F)$ for the sets of finite and archimedean places,
respectively. If $\Char(F)>0$, we put $P_\fin(F)=P(F)$.
We also use the following notation: 
\begin{itemize}
	\item $X_v:=X\otimes_FF_v$, the base change of $X$ to the completion
	$F_v$ at $v\in P(F)$;
	\item $\Sigma_\good(X) := \set{v \in P_\fin(F) | \mbox{$X$ has good reduction at $v$}}$, and
	\item $\Sigma_\bad(X) := P_\fin(F) \ssm \Sigma_\good(X)$. 
\end{itemize}

\subsection{Hasse principles}\label{subsec:Bloch}
Let $V(X)_\tor$ denote the torsion subgroup of $V(X)$. As noted in the
Introduction, Bloch's conjecture predicts that $V(X)=V(X)_\tor$. 
The local reciprocity map 
\begin{equation}
	\label{eq:tau_good}
\tau_{X_v}\colon V(X_v)_\red \hookrightarrow \piab(X_v)^{\geo} \simeq T(X_v)_{G_{F_v}}
\end{equation}
is injective and has finite image (\autoref{thm:BS}). 
Here, $T(X_v)$ is defined in \eqref{eq:TX}. 
\begin{prop}[{\cite[Section~5, Proposition~5]{KS83b}}]
\label{prop:KS}
\begin{enumerate}
	\item The group
	\[
	T(X)_{G_F}\simeq\piabXgeo
	\]
	is finite. Moreover, for almost all places $v\in P(F)$, the groups
	\[
	T(X)_{G_{F_v}}
	\simeq T(X_v)_{G_{F_v}}
	\simeq\piab(X_v)^\geo
	\]
	are finite. 
	\item Let $m$ be a nonzero integer which annihilates the finite group $T(X)_{G_F}$. 
	Then there is an exact sequence
	\[
	V(X) \xrightarrow{\loc} \bigoplus_{v\in P(F)} V(X_{v})/m \to T(X)_{G_F}\to 0.
	\]
\end{enumerate}
\end{prop}
Fix a prime $l\neq\Char(F)$.
The absolute Galois group $G_F = \Gal(F^{\sep}/F)$ acts on the $l$-torsion subgroup $J[l]$ of the Jacobian variety $J = \Jac_X$ of $X$.
Let
\[
\rhobar_{J,l}\colon G_F \to \Aut(J[l]) \simeq GL_{2g}(\Fl)
\]
be the mod-$l$ Galois representation of $J$, 
where the isomorphism on the right depends on the choice of an $\Fl$-basis of $J[l]$.
Let $G  = \Im(\rhobar_{J,l}) \subset GL_{2g}(\Fl)$. 
Following \cite[Chapter~I, Section~4]{MilADT}, 
for a finite $G$-module $M$, we define 
\[
\Sha^i(G,M) = \Ker\left(\loc^i\colon H^i(G, M)\to \prod_{v\in P(F)}H^i(G_v,M)\right), 
\]
where $G_v$ denotes the image in $G$ of a decomposition group of $G_F$ at $v$.
\begin{lem}
	\label{thm:pdiv}
	Let $l$ be a prime with $l\neq\Char(F)$. 
	If $\Sha^1(G,J[l]^\vee)=0$, then 
	the natural map 
\begin{equation}
	\label{seq:Bp}
\locbar_{l}\colon V(X)/l \to \bigoplus_{v\in P(F)} V(X_v)/l
\end{equation}
is injective.
\end{lem}
\begin{proof}
The injective Galois symbol maps give a commutative square
\[
\xymatrix{
V(X)/l\ar[r]^-{\locbar_l}\ar@{^{(}->}[d]_{s_{X,l}}&
\displaystyle\prod_{v\in P(F)}V(X_v)/l\ar@{^{(}->}[d]^{\prod s_{X_v,l}}\\
H^2(F,J[l](1))\ar[r]^-{\loc_l^2}&
\displaystyle\prod_{v\in P(F)}H^2(F_v,J_v[l](1)).
}
\]
By global Tate duality, $\Ker(\loc_l^2)$ is Pontryagin dual to
\[
\Sha^1(F,J[l]^\vee):=
\Ker\left(H^1(F,J[l]^\vee)\to
\prod_{v\in P(F)}H^1(F_v,J_v[l]^\vee)\right).
\]
Put $K=F(J[l])$. Inflation--restriction, together with the injectivity of
\[
H^1(K,J[l]^\vee)\to\prod_{w\in P(K)}H^1(K_w,J[l]^\vee),
\]
which follows from Chebotarev since $G_K$ acts trivially on $J[l]^\vee$, gives an injection
\[
\Sha^1(F,J[l]^\vee)\hookrightarrow
\Sha^1(\Gal(K/F),J[l]^\vee)=\Sha^1(G,J[l]^\vee).
\]
The group on the right vanishes by assumption; hence $\loc_l^2$, and therefore $\locbar_l$, is injective. Finally, the image of $\loc_l^2$ lies in the direct sum of the local cohomology groups by \cite[Chapter~I, Lemma~4.8]{MilADT}, so the image of $\locbar_l$ lies in $\bigoplus_vV(X_v)/l$.
\end{proof}
\begin{prop}\label{prop:Sha}
\begin{enumerate}
	\item If $X = E$ is an elliptic curve, then $\Sha^1(G,E[l]^\vee) = 0$. 
	\item If the image of $\rhobar_{J,l}$ 
	is trivial, or 
	contains a matrix $\lambda I$ for some $\lambda \neq 1$ in $\Fl^\times$, 
	where $I$ is the identity matrix in $GL_{2g}(\Fl)$, 
	then $\Sha^1(G,J[l]^\vee)=0$.
\end{enumerate}
\end{prop}
\begin{proof}
(i) By Ramakrishnan's cohomological Hasse principle \cite[Proposition~1.2.1]{Ram25}, restriction
\[
H^1(G,E[l]^\vee)\to\prod_{C\subset G}H^1(C,E[l]^\vee)
\]
is injective, where $C$ runs through the cyclic subgroups of $G$. By Chebotarev, every cyclic subgroup occurs, up to conjugacy, as the decomposition group of an unramified place. Hence the localization map is injective and $\Sha^1(G,E[l]^\vee)=0$.

(ii) The assertion is clear if $G$ is trivial. If $G$ contains $\lambda I$ with $\lambda\ne1$, then the central element $\lambda I$ acts on $J[l]^\vee$ by $\lambda^{-1}$, and $\lambda^{-1}-1$ is invertible. Sah's lemma \cite[Lemma~A.2]{BR03} gives $H^1(G,J[l]^\vee)=0$, hence $\Sha^1(G,J[l]^\vee)=0$.
\end{proof}

\subsection{Main theorems}\label{subsec:main}
For $v\in P_\fin(F)$, we choose a proper regular model over $\mathcal O_v$
and denote its special fiber by $\Xbar_v$.
\color{black}For each finite place $v$, let $\dbar_{X_v,l}$ be the local boundary map \eqref{eq:dbarlocal}. 
For every prime $l\neq \Char(F)$, define 
\begin{equation}\label{def:dgood}
	\dbar_{X,l} \colon V(X)/l \xrightarrow{\locbar_l}  \bigoplus_{v\in P_\fin(F)}V(X_v)/l\xrightarrow{\oplus\dbar_{X_v,l}} \bigoplus_{v\in P_\fin(F)}A_0(\Xbar_v)/l.
\end{equation}
The map $\locbar_l$ takes values in the direct sum, as shown
in the proof of \autoref{thm:pdiv}.
\color{black}
Projecting onto the components corresponding to the good places gives
the following commutative diagram:
\begin{equation}\label{diag:dgood}
\vcenter{
\xymatrix{
V(X)/l \ar@{=}[d]\ar[r]^-{\dbar_{X,l}} & \ds\bigoplus_{v\in P_\fin(F)}A_0(\Xbar_v)/l \ar@{->>}[d]^{\mathrm{projection}} \\ 
V(X)/l \ar[r]^-{\dbar_{X,l}^\good} & \ds\bigoplus_{v\in \Sigma_\good(X)}\Jbar_v(\F_v)/l, 
}}
\end{equation}
where $J_v:=\Jac_{X_v}=J\otimes_FF_v$. 
The snake lemma gives the following exact sequence.
\begin{lem}\label{lem:ddg}
For every prime $l\neq \Char(F)$, 
there is an exact sequence
\[	0\to \Ker(\dbar_{X,l})\to \Ker(\dbar_{X,l}^\good) \to \bigoplus_{v\in \Sigma_\bad(X)}A_0(\Xbar_v)/l
	\to \Coker(\dbar_{X,l})\to \Coker(\dbar_{X,l}^\good)\to 0.
\]
\end{lem}
Let $J[l]_{G_F}$ denote the maximal $G_F$-coinvariant quotient
\[
J[l]_{G_F}:=J[l]/\color{black}I_{G_F}(J[l])\color{black},
\]
where 
$I_{G_F}(J[l])$ is the subgroup of $J[l]$ generated by elements of the form $\sigma P - P$ 
for $\sigma \in G_F$ and $P\in J[l]$. 
For an odd prime $l$, the kernel and cokernel of
$\dbar_{X,l}^\good$ are described as follows.
\begin{thm}\label{thm:main}
	Let $l$ be an odd prime with $l\neq \Char(F)$. 
	Define
	\[
 V(X,l):=\begin{cases}\ds \bigoplus_{v \in \Sigma_\good(X), v\mid l}\Ker(\dbar_{X_v,l}) \oplus  \bigoplus_{v \in \Sigma_\bad(X)}V(X_v)/l,  & \mbox{if $\Char(F) = 0$},\\
 	\ds \bigoplus_{v \in \Sigma_\bad(X)}V(X_v)/l, &   \mbox{if $\Char(F) >0$}.\\
 \end{cases}
\] 
Assume that the localization map
$\locbar_l\colon V(X)/l\to\bigoplus_{v\in P(F)}V(X_v)/l$ is injective. 
Then there is an exact sequence
\begin{equation}\label{ex:main}
	0 \to \Ker(\dbar_{X,l}^\good)\to V(X,l) \to J[l]_{G_F}\to \Coker(\dbar_{X,l}^\good)\to  0
\end{equation}
of finite-dimensional $\F_l$-vector spaces.
\end{thm}
\begin{proof}
Apply Bloch's exact sequence in \autoref{prop:KS}(ii) 
with $m = l\cdot \#(T(X)_{G_F})$, which annihilates the finite group $T(X)_{G_F}$, 
and passing to the quotient modulo $l$, 
we obtain an exact sequence  
\begin{equation}\label{seq:Bp2} 
 V(X)/l \xrightarrow{\locbar_{l}}  \bigoplus_{v \in P(F)} V(X_v)/l  \to  T(X)_{G_F}/l \to 0.
\end{equation}
By assumption, $\locbar_l$ is injective. Since $T_l(X) \simeq T_l(J)$ and $T_l(J)/l\simeq J[l]$, there is a right exact sequence 
\[
T_l(J)_{G_F}\xrightarrow{l} T_l(J)_{G_F} \to J[l]_{G_F} \to 0 
\]
and hence
$T_l(J)_{G_F}/l\simeq J[l]_{G_F}$. 
Consider the following commutative diagram: 
\begin{equation}
\label{diag:locdbar}	
\vcenter{
\xymatrix{
0\ar[r] &  V(X)/l\ar@{=}[d] \ar[r]^-{\locbar_{l}} & \ds \bigoplus_{v \in P(F)} V(X_v)/l\ar[d]^{\oplus \dbar_{X_v,l}}  \ar[r] &  J[l]_{G_F} \ar[r]& 0,\\
&  V(X)/l\ar[r]^-{\dbar_{X,l}^\good} & \ds \bigoplus_{v\in \Sigma_\good(X)}\Jbar_{v}(\F_v)/l. & 
}}
\end{equation}
Since the local boundary map $\dbar_{X_v,l}$ is surjective for $v\in \Sigma_\good(X)$, 
the snake lemma gives the following exact sequence:
\begin{align*}
	0\to \Ker(\dbar_{X,l}^\good) &\to \bigoplus_{v\in \Sigma_\good(X)}\Ker(\dbar_{X_v,l}) \oplus  \bigoplus_{v\in \Sigma_\bad(X)\cup P_{\infty}(F)}V(X_v)/l\\ 
	&\to J[l]_{G_F}\to \Coker(\dbar_{X,l}^\good) \to  0.
\end{align*}
For each $v\in P(F)$, 
the Galois symbol map  
\[
s_{X_v,l}\colon V(X_v)/l \hookrightarrow H^2(F_v,J_v[l](1))
\]
is injective (cf.~\eqref{eq:sFl}) 
and the latter Galois cohomology group is finite. 
For every place $v\in \Sigma_\good(X)$, 
we have $\Ker(\dbar_{X_v,l}) = 0$ if $\Char(\F_v) \neq l$ (\autoref{lem:local0}). 
It remains to treat the archimedean places when $\Char(F)=0$. 
At a complex place $v$, 
$V(X_v)/l =0$ because $V(X_v) = V(X_\C)$ is $l$-divisible (\cite[Lemma~1.1]{Ras90}). 
At a real place $v$, the composition 
\[
V(X_{\R})/l \xrightarrow{\res_{\C/\R}} V(X_{\C})/l \xrightarrow{N_{\C/\R}}V(X_\R)/l
\]
is multiplication by $2 = [\C:\R]$,  which is an automorphism of
$V(X_\R)/l$ because $l$ is odd.
Since $V(X_\C)/l = 0$, this implies that $V(X_\R)/l = 0$. 
We obtain the exact sequence \eqref{ex:main}.
\end{proof}
\begin{cor}\label{cor:almostall}
Let $F$ be a global field, and let $X$ be a smooth projective curve over
$F$ with $X(F)\neq\emptyset$. Put $J=\Jac_X$. Assume that
$\End_{\overline F}(J)=\Z$ and that $J$ has semistable reduction of toric
dimension one at some finite place of $F$. Then
\[
\dbar_{X,l}^{\good}\colon V(X)/l
\to
\bigoplus_{v\in\Sigma_\good(X)}\Jbar_v(\F_v)/l
\]
is an isomorphism for all but finitely many primes
$l\neq\Char(F)$.
\end{cor}
\begin{proof}
By the big-monodromy theorem, the image of $\rhobar_{J,l}$ contains
$\Sp_{2g}(\Fl)$ for all but finitely many primes $l$
(\cite[Theorem~1]{Hal11}, \cite[Theorem~3.6]{AGW13}). For every sufficiently large prime $l$, the image contains $-I$. Hence \autoref{prop:Sha} and \autoref{thm:pdiv} show that
\[
\locbar_l\colon V(X)/l\to
\bigoplus_{v\in P(F)}V(X_v)/l
\]
is injective. Moreover, $-I-I=-2I$ is invertible on $J[l]$, so
$J[l]_{G_F}=0$.
It remains to show that $V(X,l)=0$ for all but finitely many $l$. For
each $v\in\Sigma_\bad(X)$, the reduced group $V(X_v)_\red$ is finite
by \autoref{thm:BS}. Since $V(X_v)_\div$ is divisible, we have
\[
V(X_v)/l\simeq V(X_v)_\red/l,
\]
and this group vanishes unless $l$ divides the order of
$V(X_v)_\red$. Since there are only finitely many bad places, all the
bad-reduction terms vanish for all but finitely many $l$.
If $\Char(F)>0$, this proves the assertion. Suppose that
$\Char(F)=0$. For a place $v\mid l$ at which $X$ has good reduction,
we have
$e_{F_v/\Q_l}\le [F:\Q]$.
Thus, for $l>[F:\Q]+1$, the inequality
$e_{F_v/\Q_l}<l-1$ holds, and \autoref{lem:local0}~(a) gives
$\Ker(\dbar_{X_v,l})=0$. Therefore $V(X,l)=0$ for all but finitely
many $l$ in the number field case as well. The exact sequence in
\autoref{thm:main} now shows that both the kernel and cokernel of
$\dbar_{X,l}^{\good}$ vanish.
\end{proof}

\subsection{The elliptic-curve case}\label{subsec:Gal}
We now specialize to the case in which $X=E$ is an elliptic curve over
$F$. 
The natural action of $G_F$ on $E[l]$ gives rise to a mod-$l$ Galois representation
\[
	\rhobar_{E,l}\colon G_F\to \Aut(E[l]) \simeq GL_2(\Fl). 
\]  
The isomorphism on the right depends on the choice of an $\Fl$-basis of
$E[l]$. 
\begin{lem}\label{lem:surj}
Let $l\neq\Char(F)$ be an odd prime. Assume that one of the following
conditions holds:
\begin{enumerate}[label=$(\alph*)$]
    \item There exists a basis of $E[l]$ such that the image of
    \[
        \rhobar_{E,l}\colon G_F\to
        \Aut(E[l])\simeq GL_2(\Fl)
    \]
    contains $SL_2(\Fl)$.
    \item There exists a finite place $v$ of $F$ such that $E$ has good
    reduction at $v$, $\Char(\F_v)\neq l$, and
    $l\nmid \#\Ebar_v(\F_v)$.
\end{enumerate}
Then $E[l]_{G_F}=0$.
\end{lem}
\begin{proof}
\color{black}Recall that 
$E[l]_{G_F}=E[l]/I_{G_F}(E[l])$,
where 
$I_{G_F}(E[l])$ is the subgroup of $E[l]$ generated by elements of the form $\sigma P - P$ 
for $\sigma \in G_F$ and $P\in E[l]$. \color{black}
Under (a), the image contains $-I$, and therefore
\[
(-I-I)E[l]=-2E[l]=E[l]\subset I_{G_F}(E[l]),
\]
so $E[l]_{G_F}=0$. Under (b), let
$P_v(T)=T^2-a_vT+q_v$ be the characteristic polynomial of Frobenius at $v$. Since
\[
\det\bigl(I-\rhobar_{E,l}(\Frob_v)\bigr)
\equiv P_v(1)=\#\Ebar_v(\F_v)\not\equiv0\pmod l,
\]
the endomorphism $\rhobar_{E,l}(\Frob_v)-I$ of $E[l]$ is invertible. Its image is contained in $I_{G_F}(E[l])$, and hence $E[l]_{G_F}=0$.
\end{proof}
\begin{rem}\label{rem:rho}
If $\Char(F)=0$ and $E$ is non-CM, 
then the representation $\rhobar_{E,l}$ is surjective for almost all primes $l$ 
(\cite[Section~4.2, Theorem~2]{Ser72}). 
If $\Char(F)>0$ and $E$ is non-isotrivial, 
it is known that the image of $\rhobar_{E,l}$ contains $SL_2(\Fl)$ for almost all primes $l\neq \Char(F)$  (\cite[Proposition~3.12]{BLV09}).
More precisely, 
there exists a positive constant $c(F)$ depending only on the genus of
the base curve of $F$ 
such that 
$\Im(\rhobar_{E,l}) \supset SL_2(\Fl)$ 
for every non-isotrivial elliptic curve $E$ over $F$ and every prime $l\ge c(F)$ with $l\neq \Char(F)$ 
(\cite[Theorem~1.1]{CH05}). 
In particular, for the rational function field $F=\F(t)$, one can take $c(F) = 15$. 
\end{rem}
After choosing a basis of \(E[l]\), the Weil pairing gives
$\det\circ \rhobar_{E,l} = \chi_l$, where 
$\chi_l\colon G_F\to \Fl^{\times}$ is the mod-$l$ cyclotomic character.
An elliptic curve $E$ admits an $F$-rational isogeny of degree $l$ 
if and only if the image $\rhobar_{E,l}(G_F)$ is contained in a Borel subgroup $\pmat{\ast & \ast \\ 0 & \ast } \subset GL_2(\Fl)$.
Indeed, if there is an $F$-isogeny $\phi\colon E\to E'$ of degree $l$, then $\Ker(\phi)\subset E[l]$ 
is a $G_F$-stable subgroup. 
A basis $\set{P,Q}$ with $0\neq P \in \Ker(\phi)$ and $Q \in E[l]\ssm \Ker(\phi)$ puts the representation in upper-triangular form.
Conversely, if $\rhobar_{E,l}(G_F)$ is contained in a Borel, 
there exists a basis $\set{P,Q}$ of $E[l]$ such that 
$C:=\braket{P}$ is a $G_F$-stable subgroup. 
Then the quotient map  $\phi\colon E\to E/C=:E'$ is an isogeny with kernel $C$. 
We consider the following conditions:
\begin{enumerate}
\color{black}
\item[($\mathrm{SC}_l$)]
$\dimFl E(F)[l]=1$, and $E[l]$ contains at least two
distinct $G_F$-stable lines.
\item[($\mathrm{B}'_l$)]
$\dimFl E(F)[l]=1$, and $E[l]$ contains exactly one
$G_F$-stable line.
\color{black}
	\item[$(\mathrm{B}_l)$] $E(F)[l]=0$ and there exists an $F$-isogeny $\phi\colon E'\to E$ of degree $l$ with $E'(F)[l] \neq 0$.
\end{enumerate}
%
%
\color{black}
The results on the image of $\rhobar_{E,l}$ in
\cite[Propositions~1.2 and~1.4]{RV01} also hold when
$l\ne\Char(F)$. Their proofs use only Galois-stable lines,
dual isogenies, the Weil pairing, and elementary representation
theory over $\Fl$, and do not require $\Char(F)=0$. \color{black}
Thus there exists a basis of $E[l]$ such that
\begin{equation}
	\label{eq:Imrho}
\rhobar_{E,l}(G_F) = \begin{cases}
	\pmat{1 & \ast \\0 & \chi_l(G_F)}, & \mbox{if ($\mathrm{B}'_l$) holds},\\
	\pmat{1 & 0 \\ 0 & \chi_l(G_F)}, & \mbox{if ($\mathrm{SC}_l$) holds},\\
	\pmat{\chi_l(G_F) & \ast \\ 0 & 1}, & \mbox{if ($\mathrm{B}_l$) holds}.
\end{cases}
\end{equation}
\begin{lem}\label{lem:quot}
Let $l$ be a prime with $l\neq \Char(F)$. 
\begin{enumerate}
	\item 
Assume $\mu_l\not\subset F$. 
Then 
\[
\dim_{\Fl}(E[l]_{G_F})= \begin{cases} 
0,& \mbox{if $(\mathrm{B}'_l)$ holds},\\ 
1,& \mbox{if $(\mathrm{SC}_l)$ or $(\mathrm{B}_l)$ holds},\\
2,& \mbox{if $E[l]\subset E(F)$}.
\end{cases}
\]
\item 
Assume $\mu_l\subset F$. 
Then 
\[
\dim_{\Fl}(E[l]_{G_F})= \begin{cases} 
1,& \mbox{if $(\mathrm{B}'_l)$,  or $(\mathrm{B}_l)$ holds},\\
2,& \mbox{if $(\mathrm{SC}_l)$ holds, or $E[l]\subset E(F)$}.
\end{cases}
\]
\end{enumerate}
\end{lem}
\begin{proof}
If $E[l]\subset E(F)$, then the representation is trivial and $E[l]_{G_F}=E[l]$. Assume now $\dim_{\F_l}E(F)[l]\le1$. Since
\[
(E[l]_{G_F})^\vee\simeq(E[l]^\vee)^{G_F},
\]
it is enough to compute the invariant subspace of the contragredient representations associated with \eqref{eq:Imrho}. In case $(\mathrm{SC}_l)$, this invariant space has dimension $1$ when $\chi_l$ is nontrivial and dimension $2$ when $\chi_l$ is trivial. In case $(\mathrm{B}_l)$, it is generated by the functional dual to the second basis vector and has dimension $1$. In case $(\mathrm{B}'_l)$, it has dimension $1$ if $\chi_l$ is trivial and is zero otherwise. These are exactly the dimensions stated above.
\end{proof}
\begin{rem}
	If $\Char(F)>0$, $\F$ is the finite constant field of $F$, and $l\neq \Char(F)$, then  
	the condition $\mu_l\subset F$ is equivalent to $l\mid (\#\F-1)$. 
\end{rem}

\section{Local computations for elliptic curves}\label{sec:local-elliptic}
\color{black}Throughout this section, let $K$ be a local field with finite residue field $\F = \F_K$, 
and put $p=\Char(\F)$. 
We use the same notation $v_K\colon K^\times \to \Z$, $\O_K$, and $\m_K$ as in \autoref{sec:local}. 
\color{black} 
Let us consider the case where $X=E$ is an elliptic curve over $K$. 
Let $\X = \E$ be the minimal proper regular model of $E$ (\cite[Chapter~IV, Theorem~4.5]{Sil151}, \cite[Chapter~9, Definition 4.34]{Liu02}).

\subsection{Split multiplicative reduction}
Suppose that $E$ has split multiplicative reduction. There exists an
element $q(E)\in K^\times$, called the \textbf{Tate period}, with
$m:=v_K(q(E))>0$, and it induces an isomorphism 
$(K^{\sep})^{\times}/q(E)^{\Z} \xrightarrow{\simeq} E(K^{\sep})$ 
of $G_{K}$-modules (\cite[Chapter~V, Theorem~5.3]{Sil151}; see also \cite[Theorem~3.6]{BLV09} for the case $\Char(K)>0$). 
By \autoref{thm:BS}, the reduced part $V(E)_{\red} = V(E)/V(E)_\div$ is finite.  
Fix a uniformizer $\pi$ of $K$. For $x\in K^\times$, let
$\ol u(x)\in\F^\times$ be the image of
$u(x):=x\pi^{-v_K(x)}\in\O_K^\times$ under the reduction map
$\O_K^\times\twoheadrightarrow\F^\times$. 
The following lemma rewrites the local computation of
\cite[Proposition~2.1]{Hir22} in terms of the $j$-invariant of $E$. 
\begin{lem}
	\label{lem:local}
	Assume that $E$ has split multiplicative reduction. 
	Then for every prime $l\neq p := \Char(\F)$, we have
	\[
	\dim_{\Fl}(V(E)/l) = \begin{cases}
		1, & \mbox{if $l\mid (\#\F-1), l\mid v_K(j(E))$ and $\ol{u}(j(E)) \in (\F^{\times})^l$},\\
		0,& \mbox{otherwise}.
	\end{cases}
	\]
\end{lem}
\begin{proof}
The local Galois-symbol computation of \cite[Proposition~2.1]{Hir22} gives
\begin{equation}\label{eq:dimVEl}
\dim_{\F_l}(V(E)/l)=
\begin{cases}
1,&\mbox{if $\mu_l\subset K$ and $q(E)\in(K^\times)^l$},\\
0,&\mbox{otherwise}.
\end{cases}
\end{equation}
Since $l\ne\Char(\F)$, the condition $\mu_l\subset K$ is equivalent to $l\mid(\#\F-1)$. It remains to express the second condition in terms of $j(E)$. The $q$-expansion
\[
j(E)=q(E)^{-1}+744+196884q(E)+\cdots
\]
implies $v_K(q(E))=-v_K(j(E))$ and $j(E)q(E)\in1+\m_K$. Put $n=\#\F-1$ and write, with respect to the fixed uniformizer $\pi$,
\[
j(E)=\zeta u\pi^{-m},\qquad q(E)=\xi v\pi^m,
\]
where $\zeta,\xi\in\mu_n$, $u,v\in1+\m_K$, and $m=v_K(q(E))$. The group $1+\m_K$ is $l$-divisible, and $j(E)q(E)\in1+\m_K$ shows that $\zeta\xi\in(\mu_n)^l$. Therefore
\[
q(E)\in(K^\times)^l
\Longleftrightarrow
l\mid m\ \text{and}\ \xi\in(\mu_n)^l
\Longleftrightarrow
l\mid v_K(j(E))\ \text{and}\ \ol u(j(E))\in(\F^\times)^l.
\]
Together with \eqref{eq:dimVEl}, this proves the assertion.
\end{proof}

\begin{rem}
	If $\Char(K) = 0$ and the extension $K/\Qp$ is abelian, then 
	by \cite[Theorem~1.2]{Asa06} we have 
	\[
	V(E)_\red \simeq \mu_M/(q(E),K^\times)_M,
	\]
	Here $(-,-)_M\colon K^\times\times K^\times \to \mu_M$ is the $M$-th Hilbert symbol, and
\[
M=\max\set{m\mid \mu_m\subset K}.
\]
\end{rem}
\begin{ex}\label{ex:q}
Consider the elliptic curve $E$ over $F= \F_5(t)$  
defined by the Legendre equation $y^2 = x(x-1)(x-t^2)$. 
It has split multiplicative reduction at $\p_0 = (t)$ 
with $p=5$ and its $j$-invariant is $j(E) = \dfrac{(t^4-t^2+1)^3}{t^4(t-1)^2(t+1)^2}$.  
For the elliptic curve $E_0 := E\otimes_F F_{\p_0}$ over the local field $F_{\p_0}$, 
$v_{\p_0}(j(E)) = -4$ and $\ubar(j(E)) = 1$ in $\F_5$. 
Among the primes $l\neq5$, only $l=2$ satisfies all three conditions
$l\mid(\#\F_5-1)$, $l\mid v_{\p_0}(j(E))$, and
$\ubar(j(E))\in(\F_5^\times)^l$. 
We obtain $V(E_0)/2 \simeq \Z/2$ and $V(E_0)/l = 0$ for $l\not\in \set{2,5}$ 
by \autoref{lem:local}. 
\end{ex}
\begin{lem}\label{lem:localA}
If $E$ has split multiplicative reduction, then 
$A_0(\Ebar) = 0$.	
\end{lem}
\begin{proof}
Put $Y=\Ebar$. The special fiber is of Kodaira type $I_n$, where
$n=-v_K(j(E))$ (\cite[Chapter~VII, Theorem~6.1]{Sil106}; cf.
\cite[Chapter~IV, Sections~8--9]{Sil151}). The node of an $I_1$ fiber
is split, and for $n>1$ the irreducible components form a cycle of
smooth rational curves whose intersection points are $\F$-rational.
Suppose first that $n=1$. The normalization
$\nu\colon\mathbb P^1_{\F}\to Y$ is an isomorphism away from the split
node, whose two preimages are $\F$-rational. Hence
$\nu_*\colon CH_0(\mathbb P^1_{\F})\to CH_0(Y)$ is surjective. Since
$CH_0(\mathbb P^1_{\F})\simeq\Z$ and $Y$ has an $\F$-rational point,
the degree map gives $CH_0(Y)\simeq\Z$.
Now suppose that $n>1$. Write
\[
Y=\bigcup_{i=1}^nY^{(i)},\qquad Y^{(i)}\simeq\mathbb P^1_{\F},
\]
and label the intersection points cyclically by $y_1,\ldots,y_n$ so
that $Y^{(i)}$ contains $y_i$ and $y_{i+1}$, with indices taken modulo
$n$. Since $CH_0(Y^{(i)})\simeq\Z$, the two rational points $y_i$ and
$y_{i+1}$ are rationally equivalent on $Y^{(i)}$, and hence in
$CH_0(Y)$. Thus all the classes $[y_i]$ coincide. If $y$ is any closed
point of $Y$, choose a component $Y^{(i)}$ containing it. Again using
$CH_0(Y^{(i)})\simeq\Z$, we obtain
\[
[y]=[\F(y):\F][y_i]=[\F(y):\F][y_1]
\]
in $CH_0(Y)$. Therefore the degree map is an isomorphism
$CH_0(Y)\simeq\Z$, and hence $A_0(Y)=0$.
\end{proof}

\subsection{Non-split multiplicative reduction}
We next suppose that $E$ has non-split multiplicative reduction. 
For the $j$-invariant $j(E)\in K^\times$, 
there exists a Tate period $q:= q(E)\in K^\times$ such that 
the Tate curve $E_{q}$ has the $j$-invariant $j(E_q) = j(E)$. 
There is an isomorphism $\psi\colon E\to E_q$ defined over a quadratic
extension $L/K$ (cf.~\cite[Chapter~V, Theorem~5.3]{Sil151}; see also \cite[Theorem~3.6]{BLV09}). 
Let $\epsilon:G_K\to \set{\pm 1}$ be the quadratic character associated with the quadratic extension 
$L/K$. 
For every odd prime $l\neq p$, the map $\psi$ induces an isomorphism of
$G_K$-modules
$E[l]\xrightarrow{\simeq}E_q[l]\otimes_{\Fl}\Fl(\epsilon)$. 
Recall that $\Fl(\epsilon)$ is $\Fl$ with $G_K$-action given by $\sigma(x) = \epsilon(\sigma)x$ for $x \in \Fl$. 
As in the proof of \autoref{lem:local}, 
there is a short exact sequence 
\[
0 \to \mu_l \to E_{q}[l]\to \Z/l\to 0,
\]
where $G_K$ acts on $\mu_l$ via the cyclotomic character $\chi_l\colon G_K\to \Fl^\times$. 
After fixing a basis of $E[l]$, the representation of $\sigma \in G_K$ on $E[l]$ is written by 
\begin{equation}
\label{eq:GKns}	
\rho(\sigma) = \begin{pmatrix}
	\epsilon(\sigma)\chi_l(\sigma) & \epsilon(\sigma)\kappa(\sigma) \\
	0 & \epsilon(\sigma)
\end{pmatrix},
\end{equation}
where $\kappa\colon G_K\to\Fl$ is the Kummer cocycle defined by $\sigma(\sqrt[l]{q}) = \zeta^{\kappa(\sigma)}\sqrt[l]{q}$ 
for a  primitive $l$-th root of unity $\zeta$. 
\begin{prop}\label{prop:ns}
	Suppose that $E$ has non-split multiplicative reduction, and let $l$
	be an odd prime with $l\neq p$. 
	If either $l\nmid (\#\F+1)$ or  $l\nmid v_K(j(E))$, then $V(E)/l = 0$. 
\end{prop}
\begin{proof}
The Galois symbol map
\[
V(E)/l\hookrightarrow H^2(K,E[l](1))
\]
is injective by \eqref{eq:sFl}. By local Tate duality,
\[
\dimFl H^2(K,E[l](1))
=
\dimFl (E[l]^\vee)^{G_K}.
\]
It is therefore enough to prove that $(E[l]^\vee)^{G_K}=0$.
Write $f=a\phi_1+b\phi_2\in E[l]^\vee$ with respect to the dual
basis corresponding to \eqref{eq:GKns}. The condition that $f$ is
$G_K$-invariant is equivalent to
$f\rho(\sigma)=f
\ (\sigma\in G_K)$.
Using \eqref{eq:GKns}, we obtain
\begin{equation}\label{eq:ns-invariant}
 a\epsilon(\sigma)\chi_l(\sigma)=a,
 \qquad
 \epsilon(\sigma)\bigl(a\kappa(\sigma)+b\bigr)=b
 \quad (\sigma\in G_K).
\end{equation}
If $a=0$, then the second equality and the nontriviality of $\epsilon$
imply $b=0$. Thus a nonzero invariant vector would necessarily satisfy
$a\neq0$, and the first equality would imply
$\chi_l=\epsilon$.
The quadratic extension over which $E$ becomes a Tate curve is
unramified; hence $\epsilon(\Frob_K)=-1$. Since
$\chi_l(\Frob_K)=\#\F$ modulo $l$, this equality is impossible when
$l\nmid(\#\F+1)$.
Now assume that $l\nmid v_K(j(E))$. Since
$v_K(q(E))=-v_K(j(E))$, the valuation of $q(E)$ is not divisible by
$l$. Consequently, the restriction of the Kummer cocycle $\kappa$ to
the inertia subgroup $I_K$ is nonzero. Choose $\tau\in I_K$ with
$\kappa(\tau)\neq0$. Because the quadratic character $\epsilon$ is
unramified and $l\neq p$, both $\epsilon$ and $\chi_l$ are trivial on
$I_K$. The second equality in \eqref{eq:ns-invariant}, applied to
$\tau$, gives
$a\kappa(\tau)=0$,
contrary to $a\neq0$. Hence $(E[l]^\vee)^{G_K}=0$, and therefore
$V(E)/l=0$.
\end{proof}
\begin{lem}
	Suppose that $E$ has non-split multiplicative reduction.
	If $l>2$, then $A_0(\Ebar)/l = 0$.
\end{lem}
\begin{proof}
	There exists an unramified quadratic extension $K'/K$ such that
	$E_{K'}$ has split multiplicative reduction (\cite[Appendix C, Theorem~14.1]{Sil106}). 
	The residue field extension $\F'/\F$ is a quadratic extension. 
	Since $\Spec(\O_{K'})\to \Spec(\O_K)$ is \'etale, 
	the base change $\E_{\O_{K'}} := \E \otimes_{\O_K}\O_{K'}$ is still regular 
	and the special fiber $\Ebar_{K'} := \E_{\O_{K'}}\otimes_{\O_{K'}}\F'$  
	coincides with the base change $\Ebar\otimes_\F\F'$ of  $\Ebar$ to $\F'$. 
	By \autoref{lem:localA} applied over $\F'$, we have
	$CH_0(\Ebar_{\F'})/l\simeq \Fl$. 
	The natural maps 
	\[
	i\colon CH_0(\Ebar) \to CH_0(\Ebar_{\F'}),\quad N\colon CH_0(\Ebar_{\F'}) \to CH_0(\Ebar)
	\]
	satisfy 
	$N\circ i = 2$ (the multiplication by $2$). 
	Since $N\circ i=2$ and $l>2$, the norm map induces a surjection
	\[
	\Fl\simeq CH_0(\Ebar_{\F'})/l\twoheadrightarrow CH_0(\Ebar)/l.
	\]
	Since $\Ebar(\F)\neq \emptyset$,  
	we have a decomposition $CH_0(\Ebar) = \Z\oplus A_0(\Ebar)$ and 
	thus $CH_0(\Ebar)/l = \Fl$ and $A_0(\Ebar)/l = 0$. 
\end{proof}

\section{Computations}\label{sec:Q}
We conclude with one example over a number field and one over a global function field.
\begin{ex}\label{ex:11a}
Consider the degree-$5$ isogeny class of conductor $11$. Let
$E^{(1)},E^{(2)},E^{(3)}$ be the curves with Cremona labels 11a1, 11a2,
and 11a3, respectively. Their Mordell--Weil groups are
\[
E^{(1)}(\Q)\simeq\Z/5,\qquad E^{(2)}(\Q)=0,\qquad E^{(3)}(\Q)\simeq\Z/5,
\]
and all three curves have split multiplicative reduction at $11$. For every odd prime $l\ne5$, the mod-$l$ representation is surjective. At $l=5$, the conditions $(\mathrm{SC}_5)$, $(\mathrm{B}_5)$, and $(\mathrm{B}'_5)$ hold for $E^{(1)},E^{(2)}$, and $E^{(3)}$, respectively. 
\color{black}These facts can be verified by the same method as \color{black} in \cite[Section~4]{Hir25}.
The local criterion of \autoref{lem:local}, the exact sequence of \autoref{thm:main}, and \cite[Lemma~4.1]{Hir25} give
\[
\begin{array}{c|c|c}
E&l&\text{boundary map }\dbar_{E,l}^{\good}\\ \hline
E^{(1)}&\text{every odd }l&\text{isomorphism}\\
E^{(2)}&l\ne5&\text{isomorphism}\\
E^{(2)}&l=5&\text{injective with cokernel }\F_5\\
E^{(3)}&\text{every odd }l&\text{isomorphism}.
\end{array}
\]
Indeed, at $l=5$ one has $V(E^{(1)}_{11})/5\simeq\F_5\simeq E^{(1)}[5]_{G_\Q}$, whereas $V(E^{(2)}_{11})/5=V(E^{(3)}_{11})/5=0$; for all other odd primes, the local terms and the coinvariant quotient vanish.
\end{ex}
\begin{ex}\label{ex:rat}
Let $F=\F_5(t)$ and
\[
E\colon y^2=x(x-1)(x-t^2).
\]
Then $E(F)_\tor\simeq(\Z/2)^2$ by \cite[Section~2]{McD18}, and
\[
\Delta(E)=t^4(t-1)^2(t+1)^2,\qquad
j(E)=\frac{(t^4-t^2+1)^3}{t^4(t-1)^2(t+1)^2}.
\]
The bad places are $(t),(t-1),(t+1)$, and $\infty$. The usual criterion involving $-c_4/c_6$ (after the change of variables $x=s^{-2}x'$, $y=s^{-3}y'$ at $\infty$, with $s=1/t$) shows that the reduction is split multiplicative at all four places. Let $l\ne5$ be odd. At the good place $v=(t-2)$, the reduction has $8$ rational points, so \autoref{lem:surj}(b) gives $E[l]_{G_F}=0$. Moreover, every bad residue field is $\F_5$, and $l\nmid4$; hence \autoref{lem:local} gives $V(E_v)/l=0$ at every bad place. Thus $V(E,l)=0$, and \autoref{thm:main} yields an isomorphism
\[
\dbar_{E,l}^{\good}\colon V(E)/l\xrightarrow{\simeq}
\bigoplus_{v\in\Sigma_{\good}(E)}\Ebar_v(\F_v)/l.
\]
Since $E$ has split multiplicative reduction at every bad place,
\autoref{lem:localA} gives $A_0(\Ebar_v)=0$ for
$v\in\Sigma_{\bad}(E)$. Therefore, this is equivalent to the
isomorphism
\[
\dbar_{E,l}\colon V(E)/l\xrightarrow{\simeq}
\bigoplus_{v\in P(F)}A_0(\Ebar_v)/l
\]
for every odd prime $l\ne5$.
\end{ex}

\def\cprime{$'$}
\providecommand{\bysame}{\leavevmode\hbox to3em{\hrulefill}\thinspace}
\providecommand{\MR}{\relax\ifhmode\unskip\space\fi MR }
\providecommand{\MRhref}[2]{
  \href{http://www.ams.org/mathscinet-getitem?mr=#1}{#2}
}
\providecommand{\href}[2]{#2}

\end{document}